\documentclass[10pt,reqno]{amsart}
\usepackage{mathtools}
\usepackage{amssymb}
\usepackage{amsfonts}
\usepackage[all]{xy}
\usepackage[english]{babel}
\usepackage{booktabs,multirow,graphicx}
\usepackage[notref,notcite,final]{showkeys}
\usepackage{enumerate}
\usepackage{verbatim}
\usepackage{pgf,tikz}
\usepackage{mathrsfs}
\usetikzlibrary{arrows}
\usepackage{graphicx}

\renewcommand{\phi}{\varphi}

\theoremstyle{plain}
\newtheorem{theorem}{Theorem}
\newtheorem*{theorem*}{Theorem}
\newtheorem{lemma}{Lemma}
\newtheorem{corollary}{Corollary}
\newtheorem{definition}{Definition}
\theoremstyle{remark}
\newtheorem{remark}{Remark}
\begin{document}
\definecolor{cqcqcq}{rgb}{0.7529411764705882,0.7529411764705882,0.7529411764705882}
\title{Iteration problem for distributional chaos}

\author{Jana Hant\'akov\'a}

\address{ Mathematical Institute, Silesian University, CZ-746 01 
Opava, Czech Republic}

\email{jana.hantakova@math.slu.cz}

\thanks{The research was supported by grant SGS/18/2016 from  the Silesian University in Opava. Support of this institution is gratefully acknowledged.}

\begin{abstract} We disprove the conjecture that distributional chaos of type 3 (briefly, DC3) is iteration invariant and show that a slightly strengthened definition, denoted by DC2$\frac{1}{2}$, is preserved under iteration, i.e. $f^n$ is  DC2$\frac{1}{2}$ if and only if $f$ is too. Unlike DC3, DC2$\frac{1}{2}$ is also conjugacy invariant and implies Li-Yorke chaos. The definition of DC2$\frac{1}{2}$ is the following: a pair $(x,y)$ is DC2$\frac{1}{2}$ iff $\Phi_{(x,y)}(0)<\Phi^*_{(x,y)}(0)$, where $\Phi_{(x,y)}(\delta)$ (resp. $\Phi^*_{(x,y)}(\delta)$) is lower (resp. upper) density of times $k$ when $d(f^k(x),f^k(y))<\delta$ and  both densities are defined at 0 as limits of their values for $\delta\to 0^+$. Hence DC$2\frac{1}{2}$ shares similar properties with DC1 and DC2 but unlike them, strict DC$2\frac{1}{2}$ systems must have zero topological entropy.

{\small {2000 {\it Mathematics Subject Classification.}}
Primary 37D45; 37B40.
}
\end{abstract}

\maketitle
\pagestyle{myheadings}
\markboth{J. Hant\'akov\'a}
{Iteration problem for DC3}
\section{Introduction}
The study of chaotic pairs in dynamics started with Li and Yorke \cite{LY}, who studied pairs of points with the property that their orbits are neither asymptotic nor separated by any positive fixed constant. Schweizer and Sm\' ital \cite{SchSm} introduced the related concept of a distributionally chaotic pair as two points for which the statistical distribution of distances between the orbits does not converge. The existence of a single distributionally chaotic pair is equivalent to the positivity topological entropy (and some other notions of chaos) when restricted to the compact interval case. \\
Later, distributional chaos was divided into three types, DC1, DC2, and DC3, see \cite{BSS}. Relations between them and the relation between distributional chaos and Li-Yorke chaos are investigated by many authors, see e.g. \cite{BSS,O}. One can easily see from the definitions that DC1 implies DC2 and DC2 implies DC3. On the other hand, there are examples which show that DC1 is stronger than DC2 and DC2 is stronger than DC3. It is also obvious that either DC1 or DC2 implies Li-Yorke chaos. While it is proved in \cite{Li,W,SS} that DC1 and DC2 are well-defined properties of a dynamical system, DC3 appears to be very weak and unstable. Hence we propose to replace DC3 by a slightly strengthened definition, denoted by DC2$\frac{1}{2}$.\\

Recently, Li in \cite{Li} and Wang et al. in \cite{W} proved that DC1 and DC2 are iteration invariants and posed an open question whether DC3 is also preserved under iteration. Dvo\v rakov\'a proved in \cite{JD} one implication -  if a function $f$ is DC3, then $f^n$ is DC3, for every $n\in\mathbb {N}$, and conjectured that the opposite implication also holds. We disprove this conjecture by finding a dynamical system which has a DC3 pair with respect to $f^2$ but no DC3 pairs with respect to $f$.\\

It is proved in \cite{BSS} that DC3 does not imply chaos in the sense of Li and Yorke and it is not invariant with respect to topological conjugacy. Hence the definition of DC3 was strengthened in such a way that it is preserved under conjugacy and implies Li-Yorke chaos, but is still weaker than DC2 -- the new definition was denoted by DC2$\frac{1}{2}$ (see \cite{DRR}). The only difference between DC2$\frac{1}{2}$ and DC3 is the following: a pair $(x,y)$ is DC3 iff $\Phi(\delta)<\Phi^*(\delta)$, for every $\delta$ in some interval $I$. We say that a pair $(x,y)$ is DC2$\frac{1}{2}$ iff $\Phi(0)<\Phi^*(0)$, where the distribution functions at 0 are defined as limits of their values for $\delta\to 0^+$. This change in definition ensures that DC2$\frac{1}{2}$ is conjugacy invariant, implies Li-Yorke chaos and we will show in this paper that it is (like DC1 and DC2) iteration invariant.\\

We call a DC$i$ system strict if it possesses no distributionally chaotic pairs of types smaller than $i$. By results of \cite{DL}, positive topological entropy implies existence of DC2 pairs, therefore strict DC2$\frac{1}{2}$ systems must have zero topological entropy. Another strengthened distributional chaos, denoted by DC1$\frac{1}{2}$, was proposed by authors in \cite{DL}. DC1$\frac{1}{2}$ chaos is stronger than $DC2$ and is implied by positive topological entropy.\\

However, it should be noticed that the distributional chaos in cited \cite{Li,W,JD} was defined as the existence of a single distributionally scrambled pair, but nowadays it is generally assumed that distributional chaos means the existence of an uncountable distributionally scrambled set. That arises a natural question - are results for the existence of uncountable chaotic sets for iterates of a function the same as for the existence of pairs? The answer to this question strongly depends on the type of distributional chaos.\\

The paper is organized as follows: first two sections are introductory. In the third we show that distributional chaos of type 3 is not iteration invariant by creating a counterexample. The fourth section investigates distributional chaos of type 2$\frac{1}{2}$ and proves that it is iteration invariant. The fifth section discusses whether the existence of an infinite or an uncountable distributionally scrambled set is preserved under iteration.

\section{Terminology}
Let $(X,d)$ be a non-empty compact metric space.  A pair $(X,f)$, where $f$ is a continuous self-map acting on $X$, is called a \emph{topological dynamical system}.  We define the \emph{forward orbit} of $x$, denoted by $Orb^+_f(x)$ as the set $\{f^n(x):n\geq0\}$. We say that pair $(x,y)$ is \emph{asymptotic} if $\lim_{n\to\infty}d(f^i(x),f^i(y))=0$ or \emph{eventually equal} if there is $j\in\mathbb{N}$ such that $f^j(x)=f^j(y)$. 

\begin{definition}
For a pair $(x_1, x_2)$ of
points in $X$, define the \emph{lower distribution function} generated by $f$ as
$$\Phi_{(x_1, x_2)}(\delta)=\displaystyle\liminf_{m\to\infty}\frac{1}{m}\#\{0 \le k \le m;d(f^k(x_1),f^{k}(x_2))<\delta\},$$
and the \emph{upper distribution function} as 
$$\Phi^*_{(x_1, x_2)}(\delta)=\displaystyle\limsup_{m\to\infty}\frac{1}{m}\#\{0 \le k \le m;d(f^k(x_1),f^{k}(x_2))<\delta\},$$
where $\#A$ denotes the cardinality of the set $A$.\\ 
A pair $(x_1, x_2)\in X^2$ is called 
\emph{distributionally chaotic of type 1 (briefly DC1)} if 
$$\Phi^*_{(x_1, x_2)}(\delta)=1, \mbox{  for every $0<\delta\le \text{diam }X$}$$
  and  
  $$ \Phi_{(x_1, x_2)}(\epsilon)=0, \mbox{  for some  }0<\epsilon\le \text{diam }X ,$$
\emph{distributionally chaotic of type 2 (briefly DC2)} if 
$$\Phi^*_{(x_1, x_2)}(\delta)=1, \mbox{  for every $0<\delta\le \text{diam }X$}$$
  and  
  $$ \Phi_{(x_1, x_2)}(\epsilon)< 1,\mbox{  for some  }0<\epsilon\le \text{diam }X ,$$
\emph{distributionally chaotic of type 2$\frac{1}{2}$ (briefly DC2$\frac{1}{2}$)} if there exist numbers $c,q>0$ such that
  $$\Phi_{(x_1, x_2)}(\delta)<c< \Phi^*_{(x_1, x_2)}(\delta),  \mbox{  for every $0<\delta\le q,$}$$
\emph{distributionally chaotic of type 3 (briefly DC3)} if $$ \Phi_{(x_1, x_2)}(\delta)<\Phi^*_{(x_1, x_2)}(\delta), \mbox{  for every $\delta\in (a,b),$ where }0\leq a<b\le \text{diam }X.$$
The dynamical system $(X,f)$ is \emph{distributionally chaotic of type $i$} (DC$i$ for short), where $i=1,2,2\frac{1}{2},3$, if there is an uncountable set $S\subset X$ such that any pair of
distinct points from $S$ is distributionally scrambled of type $i$.
\end{definition}

We can define both distribution functions at $0$ as the limit $\Phi_{(x_1, x_2)}(0)=\lim_{\delta\to0^+}\Phi_{(x_1, x_2)}(\delta)$ and $\Phi^*_{(x_1, x_2)}(0)=\lim_{\delta\to 0^+}\Phi^*_{(x_1, x_2)}(\delta)$. Then $(x_1,x_2)$ being DC1 is equivalent to

$$\Phi^*_{(x_1, x_2)}(0)=1, \Phi_{(x_1, x_2)}(\epsilon)=0, \mbox{  for some  }0<\epsilon\le \text{diam }X ;$$
DC2 is equivalent to

$$\Phi^*_{(x_1, x_2)}(0)=1, \Phi_{(x_1, x_2)}(0)< 1;$$
DC$2\frac 12$ is equivalent to

$$\Phi_{(x_1, x_2)}(0)<\Phi^*_{(x_1, x_2)}(0).$$

\section{Iteration problem for DC3}
\begin{theorem} Distributional chaos of type 3 with respect to $f^2$ doesn't imply distributional chaos of type 3 with respect to $f$.
\end{theorem}
Proof of this theorem consists of finding a dynamical system which has a DC3 pair with respect to $f^2$ but no DC3 pairs with respect to $f$. The main obstacle in creating such system is that by \cite{Li}, a pair is DC2 with respect to $f$ iff it is DC2 with respect to $f^2$, hence the desired system has to be without any DC2 pairs. There are only few such examples in the literature (see \cite{BSS}, \cite{Li}, \cite{O}). \\
In this paper we will gradually modify a very simple dynamical system from Section 3.1 to get a DC3 system in Section 3.2 and then prove our theorem in Section 3.3.

\subsection{Oscillator}
Our first goal is to construct an oscillatoric dynamical system, where points regularly move from the right endpoint of some interval to the left endpoint (and back).
Let $I$ be the unit interval and $g_m:I\rightarrow I$ be a mapping defined as 
\begin{equation}\label{definition_g} g_m (x)= \left\{
  \begin{array}{l l}
    0&\quad 0\leq x<\frac{1}{m}\\
    x-\frac{1}{m}& \quad \frac{1}{m}\leq x \leq 1\\
    
  \end{array} \right.
  \end{equation}
  and $\hat{g}_m:I\rightarrow I$ defined as
  \begin{equation}\label{definition_hg}\hat{g}_m (x)= \left\{
  \begin{array}{l l}
    x+\frac{1}{m}&\quad 0\leq x<1-\frac{1}{m}\\
    1& \quad 1-\frac{1}{m}\leq x \leq 1.\\
    
  \end{array} \right.
  \end{equation}
Dynamical system $O_1$ consists of a compact metric space $I \times (\{\frac{1}{k}:k\in\mathbb{N}\}\cup\{0\})$ endowed with max-metric and a function $F$ such that, for $x\in I$,
$$F([x,0])=[x,0]$$
$$F([x,\frac{1}{k}])=[f_k(x),\frac{1}{k+1}],$$
with 
 \begin{equation}\label{definition_f}f_k= \left\{
  \begin{array}{l l l}
    g_m&\quad s_m+2im<k\leq s_m+2im+m&\\
    \hat{g}_m& \quad s_m+(2i+1)m<k\leq s_m+(2i+2)m&\quad i\in\{0,1,\ldots,n_m-1\},\\
    
  \end{array} \right.
  \end{equation}
  where $s_m=n_1\cdot 2 \cdot 1+n_2\cdot 2\cdot 2+\ldots+n_{m-1}\cdot 2\cdot (m-1)$ and $\{n_m\}_{m=1}^{\infty}$ is an increasing sequence of integers which will be specified later. Notice that point $[1,1]$ moves from left to the right applying $m$-times $g_m$ and then from left to right applying $m$-times $\hat{g}_m$ and repeat this movement $n_m$-times in each time interval $(s_m,s_{m+1})$. Other points are fixed, lie on the orbit of $[1,1]$ or are eventually mapped on the orbit of $[1,1]$. We will show that $O$ is not distributionally chaotic - since points on the orbit of $[1,1]$ are asymptotic to $[1,1]$, it is enough to show that $(x,y)$ is not DC3, where $x=[1,1]$ and $y=[z,0]$, where $z\in I$. Because the second coordinate of $x$ decreases with time to zero and we are considering max-metric, it is sufficient to prove 
  $$\lim_{n\to\infty}\frac{1}{n}\#\{i\leq n: |f^i(1)-z|<\delta\}$$
  exists, where $f^i=f_i\circ f_{i-1}\circ\ldots\circ f_2\circ f_1$.\\
  Denote $J_{\delta}=(y-\delta,y+\delta)\cap I$ and calculate, how many times point hits fixed subinterval $J_{\delta}$ if it oscillates with velocity $\frac{1}{m}$ for $2m$ times between right endpoint and left endpoint of $I$. Denote the number of hitting times by $P_m$. We estimate $P_m$ by 
  \begin{equation}\label{es}|J_{\delta}|\cdot 2m-2\leq P_m\leq |J_{\delta}|\cdot 2m+2.\end{equation}
For every $n\in\mathbb{N}$, there is $m\in\mathbb{N}$ such that $$n=s_m+2m\alpha+\beta,$$ where $0\leq \alpha<n_m$ and $0\leq \beta<2m.$ Since $$\#\{i\leq n: |f^i(1)-z|<\delta\}=P_1n_1+P_2n_2+\ldots+P_{m-1}n_{m-1}+P_m\alpha+\gamma,\quad 0\leq\gamma\leq\beta,$$
we can estimate the expression according to (\ref{es})  from bellow by $$(|J_{\delta}|\cdot 2\cdot 1-2)n_1+(|J_{\delta}|\cdot 2\cdot 2-2)n_2+\ldots+(|J_{\delta}|\cdot 2\cdot m-2)\alpha$$ and from above by
$$(|J_{\delta}|\cdot 2\cdot 1+2)n_1+(|J_{\delta}|\cdot 2\cdot 2+2)n_2+\ldots+(|J_{\delta}|\cdot 2\cdot m+2)\alpha+2m.$$
We define the sequence $\{n_i\}_{i=1}^{\infty}$ in such a way that $\lim_{i\to\infty}\frac{n_{i}}{s_i}=0$ to get
\begin{multline}\lim_{m\to\infty}\frac{1}{s_m+2m\alpha+\beta}(|J_\delta|\cdot(s_m+2m\alpha)-(2n_1+2n_2+2n_{m-1}\ldots+2\alpha))=\\=\lim_{m\to\infty}\frac{1}{s_m+2m\alpha+\beta}(|J_\delta|\cdot(s_m+2m\alpha)+(2n_1+2n_2+\ldots+2n_{m-1}+2\alpha+2m))=|J_{\delta}|,\end{multline}
implying $\Phi_{(x, y)}(\delta)=\Phi^*_{(x, y)}=|J_{\delta}|.$\\
\begin{remark} The same calculation holds without the assumption $z\in I$, i.e. $\Phi_{(x, y)}(\delta)=\Phi^*_{(x, y)}=|J_{\delta}|$ for any fixed point $y=[z,0]$, $z\in\mathbb{R}$, and $x=[1,1]$.  \end{remark}
\subsection{Distributionally chaotic oscillators}
We extend the dynamical system from previous section by adding one more oscillator with distance 1 to the right side of $O_1$. Let $K$ be the interval $[2,3]$ and $h_m:K\rightarrow K$ be a mapping defined as 
\begin{equation}\label{definition_h} h_m (x)= \left\{
  \begin{array}{l l}
    2&\quad 2\leq x<2+\frac{1}{m}\\
    x-\frac{1}{m}& \quad 2+\frac{1}{m}\leq x \leq 3\\
    
  \end{array} \right.
  \end{equation}
  and $\hat{h}_m:K\rightarrow K$ defined as
  \begin{equation}\label{definition_hh}\hat{h}_m (x)= \left\{
  \begin{array}{l l}
    x+\frac{1}{m}&\quad 2\leq x<3-\frac{1}{m}\\
    3& \quad 3-\frac{1}{m}\leq x \leq 3.\\
    
  \end{array} \right.
  \end{equation}
  Dynamical system $O_2$ consists of compact metric space $K\times (\{\frac{1}{k}:k\in\mathbb{N}\}\cup\{0\})$ and a function $\hat{F}$ defined by
  \begin{equation}\label{definition_F}
  \begin{array}{l l}
    \hat{F}([x,0])=[x,0]&\quad x\in K\\
    \hat{F}([x,\frac{1}{k}])=[\hat{f}_k(x),\frac{1}{k+1}]&\quad x\in K, k \in\mathbb{N}.
  \end{array} 
  \end{equation}
The function $\hat{f}_k$ is defined for $k\in\{s_m,s_m+1,\ldots,s_{m+1}\}$ differently for even and odd $m$. For odd $m$,
   \begin{equation}\label{definition_hf}\hat{f}_k= \left\{
  \begin{array}{l l l}
    Id&\quad s_m<k\leq s_m+2m  \\
    h_m&\quad s_m+m<k\leq s_m+2m\\
    \hat{h}_m&\quad s_m+2im<k\leq s_m+2im+m  \\
    h_m& \quad s_m+(2i+1)m<k\leq s_m+(2i+2)m&\quad i\in\{1,\ldots,n_m-1\},
    
  \end{array} \right.
  \end{equation}
  for even $m$,
  \begin{equation}\label{definition_hhf}\hat{f}_k= \left\{
  \begin{array}{l l l}
    Id&\quad s_m<k\leq s_m+2m  \\
    \hat{h}_m&\quad s_m+m<k\leq s_m+2m\\
    h_m&\quad s_m+2im<k\leq s_m+2im+m  \\
    \hat{h}_m& \quad s_m+(2i+1)m<k\leq s_m+(2i+2)m&\quad i\in\{1,\ldots,n_m-1\}.
    
  \end{array} \right.
  \end{equation}
  $O_2$ is made similarly as $O_1$, we are just using $\hat{f}_k$ instead of $f_k$. For better understanding of formulas in definition of $\hat{f}_k$ see Figure \ref{fig:1}.
  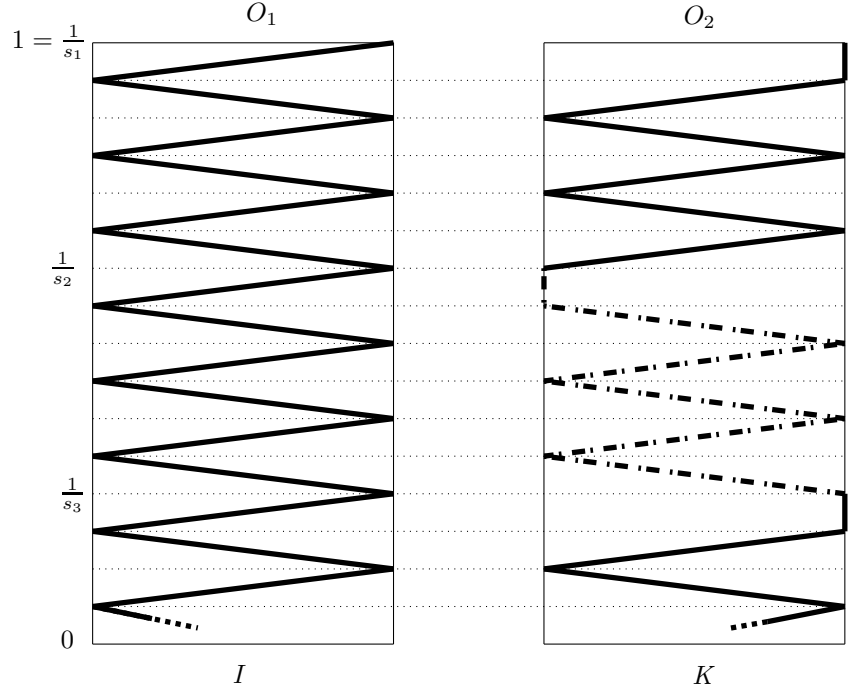
\begin{figure}[h]

\begin{tikzpicture}
\clip(-1.3,-0.7) rectangle (10.051050321158185,9.384276094597881);
\draw (0.,8.)-- (4.,8.);
\draw (4.,8.)-- (4.,0.);
\draw (0.,8.)-- (0.,0.);
\draw (0.,0.)-- (4.,0.);
\draw (6.,8.)-- (10.,8.);
\draw (10.,8.)-- (10.,0.);
\draw (10.,0.)-- (6.,0.);
\draw (6.,0.)-- (6.,8.);
\draw [dotted] (0.,7.5)-- (10.,7.5);
\draw [dotted] (0.,7.)-- (10.,7.);
\draw [dotted] (0.,6.5)-- (10.,6.5);
\draw [dotted] (0.,6.)-- (10.,6.);
\draw [dotted] (0.,5.5)-- (10.,5.5);
\draw [dotted] (0.,5.)-- (10.,5.);
\draw [dotted] (0.,4.5)-- (10.,4.5);
\draw [dotted] (0.,4.)-- (10.,4.);
\draw [dotted] (0.,3.5)-- (10.,3.5);
\draw [dotted] (0.,3.)-- (10.,3.);
\draw [dotted] (0.,2.5)-- (10.,2.5);
\draw [dotted] (0.,2.)-- (10.,2.);
\draw [dotted] (0.,1.5)-- (10.,1.5);
\draw [dotted] (0.,1.)-- (10.,1.);
\draw [dotted] (0.,0.5)-- (10.,0.5);
\draw [line width=2.pt] (4.,8.)-- (0.,7.5);
\draw [line width=2.pt] (0.,7.5)-- (4.,7.);
\draw [line width=2.pt] (4.,7.)-- (0.,6.5);
\draw [line width=2.pt] (0.,6.5)-- (4.,6.);
\draw [line width=2.pt] (4.,6.)-- (0.,5.5);
\draw [line width=2.pt] (0.,5.5)-- (4.,5.);
\draw [line width=2.pt] (4.,5.)-- (0.,4.5);
\draw [line width=2.pt] (0.,4.5)-- (4.,4.);
\draw [line width=2.pt] (4.,4.)-- (0.,3.5);
\draw [line width=2.pt] (0.,3.5)-- (4.,3.);
\draw [line width=2.pt] (4.,3.)-- (0.,2.5);
\draw [line width=2.pt] (0.,2.5)-- (4.,2.);
\draw [line width=2.pt] (4.,2.)-- (0.,1.5);
\draw [line width=2.pt] (0.,1.5)-- (4.,1.);
\draw [line width=2.pt] (4.,1.)-- (0.,0.5);
\draw [line width=2.pt,dotted] (0.,0.5)-- (1.3981410439421973,0.2145171451162878);
\draw [line width=2.pt] (0.7889770769390853,0.3389007823364577)-- (0.,0.5);
\draw [line width=2.pt] (6.,7.)-- (10.,6.5);
\draw [line width=2.pt] (10.,6.5)-- (6.,6.);
\draw [line width=2.pt] (6.,7.)-- (10.,7.5);
\draw [line width=2.pt] (6.,6.)-- (10.,5.5);
\draw [line width=2.pt] (10.,5.5)-- (6.,5.);
\draw [line width=2.pt] (10.,8.)-- (10.,7.5);
\draw [line width=2.pt,dash pattern=on 1pt off 2pt on 5pt off 4pt] (6.,5.)-- (6.,4.5);
\draw [line width=2.pt,dash pattern=on 1pt off 2pt on 5pt off 4pt] (6.,4.5)-- (10.,4.);
\draw [line width=2.pt,dash pattern=on 1pt off 2pt on 5pt off 4pt] (10.,4.)-- (6.,3.5);
\draw [line width=2.pt,dash pattern=on 1pt off 2pt on 5pt off 4pt] (6.,3.5)-- (10.,3.);
\draw [line width=2.pt,dash pattern=on 1pt off 2pt on 5pt off 4pt] (10.,3.)-- (6.,2.5);
\draw [line width=2.pt,dash pattern=on 1pt off 2pt on 5pt off 4pt] (6.,2.5)-- (10.,2.);
\draw [line width=2.pt] (10.,2.)-- (10.,1.5);
\draw [line width=2.pt] (10.,1.5)-- (6.,1.);
\draw [line width=2.pt] (6.,1.)-- (10.,0.5);
\draw [line width=2.pt] (10.,0.5)-- (8.9709778184326,0.2982609957030703);
\draw [line width=2.pt,dotted] (8.9709778184326,0.2982609957030703)-- (8.420661086005177,0.20255373788960537);
\draw (1.9189143244414386,8.647879372490642) node[anchor=north west] {$O_1$};
\draw (7.736448429088616,8.61105953638528) node[anchor=north west] {$O_2$};
\draw (1.734815143914629,-0.15206145669087243) node[anchor=north west] {\emph{I}};
\draw (7.846907937404702,-0.17047137474355342) node[anchor=north west] {\emph{K}};
\draw (-0.54550559603831754,0.29772698631006644) node[anchor=north west] {0};
\draw (-1.2,8.355829782227237) node[anchor=north west] {$1=\frac{1}{s_1}$};
\draw (-0.6952940390392555,5.336603221587554) node[anchor=north west] {$\frac{1}{s_2}$};
\draw (-0.5664246126704888,2.3357865790005523) node[anchor=north west] {$\frac{1}{s_3}$};
\end{tikzpicture}
  \caption{Movement of points $x=[1,1]$ in $O_1$ and $y=[3,1]$ in $O_2$. }
   \label{fig:1}
\end{figure}
 Adding $m$ identity mappings at the beginning of each time interval $(s_m,s_{m+1})$ causes change in the movement of $y=[3,1]$ - for $m$ odd, $y$ starts to oscillate from the right endpoint but for $m$ even, $y$ starts to oscillate from the left endpoint. Nevertheless, these identities doesn't affect the calculation of distribution functions of $y$ and some fixed point in $K$ - we get the same results as in (5). We conclude that DC3 pairs are neither in $O_1$ nor in $O_2$.\\
 Consider the union of dynamical systems $O_1\cup O_2$ defined naturally as a space $(I\cup K)\times (\{\frac{1}{k}:k\in\mathbb{N}\}\cup\{0\})$ with a function $G$ such that $G$ restricted to $I\times (\{\frac{1}{k}:k\in\mathbb{N}\}\cup\{0\})$ is equal to $F$ and $G$ restricted to $K\times (\{\frac{1}{k}:k\in\mathbb{N}\}\cup\{0\})$ is equal to $\hat{F}$. \\
 Now we investigate behavior of pairs in $O_1\cup O_2$. We have already seen that there are no DC3 pairs inside $O_1$ or  $O_2$.
By Remark 1, any fixed point in $K$ (respectively in $I$) and $y=[1,1]$ (respectively $x=[3,1]$) also can't be DC3.
All other possible pairs consist of points asymptotic or eventually equal to $x=[1,1]$ or $y=[3,1]$, so it is sufficient to examine only $\Phi_{(x, y)}$ and $\Phi^*_{(x, y)}$.\\
 In time interval $(s_m+2m, s_{m+1})$, where $m$ is even, are points $x$ and $y$ \emph{synchronic} (see the dashed part of trajectory of $y$ in Figure \ref{fig:1}) - they maintain the same distance.  If we denote the first coordinate of $G^i(x)$ by $x_i$ and the first coordinate of $G^i(y)$ by $y_i$, then
$$y_i=2+x_i, \quad\text{ for } s_m+2m<i\leq s_m+2m n_m, \quad m \text{ is even },$$
therefore $\#\{s_m+2m<i\leq s_m+2m n_m, d(F^i(x), F^i(y))<\delta\}$ is either 0, for $\delta\leq2$, or $2mn_m-2m$, for $\delta>2$.\\
Since the sequence $\{n_i\}_{i=1}^{\infty}$  grows very quickly, i.e. $\lim_{i\to\infty}\frac{n_{i}}{s_i}=0$, 
\begin{equation}\label{even} \Phi^e_{(x,y)}=\lim_{\substack{m\to\infty \\ m \text{ is even}}}\frac{1}{s_m+2mn_m}\{ 0<i\leq s_m+2mn_m, d(G^i(x), G^i(y))<\delta\}=
 \left\{
  \begin{array}{l l}
    0&\quad \delta\leq2\\
   1& \quad \delta>2.\\
    
  \end{array} \right.
  \end{equation}
In time interval $(s_m+2m, s_{m+1})$, where $m$ is odd, are points $x$ and $y$ \emph{asynchronic} (see solid parts of trajectory of $y$ in Figure \ref{fig:1}) - $x$ is on the left endpoint of its interval if $y$ is on the right endpoint of its interval (and vice versa), therefore
$$y_i=3-x_i, \quad\text{ for } s_m+2m<i\leq s_m+2m n_m, \quad m \text{ is odd }.$$
From the perspective of point $x$, $y$ is approaching to $x$ to the distance 1 and then is leaving to the distance 3 with doubled speed $\frac {2}{m}$.
 This type of movement (one point is fixed and one point is oscillating) was investigated in the previous section - see calculation between (4) - (5) and Remark 1, hence
\begin{equation}\label{odd} \Phi^o_{(x,y)}=\lim_{\substack{m\to\infty \\ m \text{ is odd}}}\frac{1}{s_m+2mn_m}\{ 0<i\leq s_m+2mn_m, d(G^i(x), G^i(y))<\delta\}=
 \left\{
  \begin{array}{l l}
    0&\quad \delta\leq1\\
   \frac{\delta-1}{2}& \quad 1<\delta\leq 3\\
   1&\quad \delta> 3.\\ 
  \end{array} \right.
  \end{equation}  
  Finally we can conclude
  $$\Phi_{(x,y)}=\min\{ \Phi^o_{(x,y)}, \Phi^e_{(x,y)}\}, \quad \Phi^*_{(x,y)}=\max\{ \Phi^o_{(x,y)}, \Phi^e_{(x,y)}\}.$$
  By (\ref{even}) and (\ref{odd}), $\Phi_{(x,y)}(\delta)<\Phi^*_{(x,y)}(\delta)$, for $\delta\in (1,3)$, hence $(x,y)$ - and all pairs consisting of points asymptotic or eventually equal to $x$ and $y$ - is DC3.   \\

 \subsection{Iteration problem}
Dvo\v r\'akov\'a in \cite{JD} proved that if $(x,y)$ is DC3 pair with respect to $G$ then there is $j\in\{0,1\}$ such that $(G^j(x),G^j(y))$ is DC3 pair with respect to $G^2$. We keep the notation from previous sections and will define a new function $H$ such that $H^{2}(z)=G^{2}(z)$, for $z\in Orb^+_G(x)\cup Orb^+_G(y)$. Hence $(G^j(x),G^j(y))$ remains DC3 pair with respect to $H^2$ but there will be no DC3 pairs with respect to $H$ which completes the proof of Theorem 1.\\

We will add one more oscillator with distance 1 to the left side of $O_1$. 
Let $I,K$ be intervals and $f_k, \hat{f}_k$ be functions from previous sections. 
Let $J$ be the interval $[-2,-1]$ and $l_m:J\rightarrow J$ be a mapping defined as 
\begin{equation}\label{definition_l} l_m (x)= \left\{
  \begin{array}{l l}
    -2&\quad -2\leq x<-2+\frac{1}{m}\\
    x-\frac{1}{m}& \quad -2+\frac{1}{m}\leq x \leq -1\\
    
  \end{array} \right.
  \end{equation}
  and $\hat{l}_m:J\rightarrow J$ defined as
  \begin{equation}\label{definition_hl}\hat{l}_m (x)= \left\{
  \begin{array}{l l}
    x+\frac{1}{m}&\quad -2\leq x<-1-\frac{1}{m}\\
    -1& \quad -1-\frac{1}{m}\leq x \leq -1.\\
 \end{array} \right.
  \end{equation}

The definition of $\tilde{f}_k$ is symmetrical to $\hat{f}_k$, we  use $l_m$ (resp. $\hat{l}_m$) instead of $\hat{h}_m$ (resp.  $h_m$).  For odd $m$,
   \begin{equation}\label{definition_tf}\tilde{f}_k= \left\{
  \begin{array}{l l l}
    Id&\quad s_m<k\leq s_m+2m  \\
    \hat{l}_m&\quad s_m+m<k\leq s_m+2m\\
    l_m&\quad s_m+2im<k\leq s_m+2im+m  \\
    \hat{l}_m& \quad s_m+(2i+1)m<k\leq s_m+(2i+2)m&\quad i\in\{1,\ldots,n_m-1\},
    
  \end{array} \right.
  \end{equation}
  for even $m$,
  \begin{equation}\label{definition_hhf}\tilde{f}_k= \left\{
  \begin{array}{l l l}
    Id&\quad s_m<k\leq s_m+2m  \\
    l_m&\quad s_m+m<k\leq s_m+2m\\
    \hat{l}_m&\quad s_m+2im<k\leq s_m+2im+m  \\
    l_m& \quad s_m+(2i+1)m<k\leq s_m+(2i+2)m&\quad i\in\{1,\ldots,n_m-1\}.
    
  \end{array} \right.
  \end{equation}

Dynamical system $O$ consists of compact metric space $(I\cup K\cup J)\times (\{\frac{1}{k}:k\in\mathbb{N}\}\cup\{0\})$ and a function $H$ defined by
  \begin{equation}\label{definition_H}
 \begin{array}{l l}
H([x,0])=[x,0]&\quad x\in I \cup K\cup J\\
H([x,\frac{1}{k}])=[f_k(x),\frac{1}{k+1}]&\quad x\in I, k \in\mathbb{N}\\
 H([x,\frac{1}{k}])=[1-\hat{f}_k(x),\frac{1}{k+1}]&\quad x\in K, k \in\mathbb{N}\\
 H([x,\frac{1}{k}])=[1-\tilde{f}_k(x),\frac{1}{k+1}]&\quad x\in J, k \in\mathbb{N}.
 \end{array} 
 \end{equation}
\begin{figure}
\begin{tikzpicture}[scale=0.9]
\clip(-6.5,-0.5) rectangle (10.5,8.5);
\draw (0.,8.)-- (4.,8.);
\draw (4.,8.)-- (4.,0.);
\draw (0.,8.)-- (0.,0.);
\draw (0.,0.)-- (4.,0.);
\draw (6.,8.)-- (10.,8.);
\draw (10.,8.)-- (10.,0.);
\draw (10.,0.)-- (6.,0.);
\draw (6.,0.)-- (6.,8.);
\draw [dotted] (0.,7.5)-- (10.,7.5);
\draw [dotted] (0.,7.)-- (10.,7.);
\draw [dotted] (0.,6.5)-- (10.,6.5);
\draw [dotted] (0.,6.)-- (10.,6.);
\draw [dotted] (0.,5.5)-- (10.,5.5);
\draw [dotted] (0.,5.)-- (10.,5.);
\draw [dotted] (0.,4.5)-- (10.,4.5);
\draw [dotted] (0.,4.)-- (10.,4.);
\draw [dotted] (0.,3.5)-- (10.,3.5);
\draw [dotted] (0.,3.)-- (10.,3.);
\draw [dotted] (0.,2.5)-- (10.,2.5);
\draw [dotted] (0.,2.)-- (10.,2.);
\draw [dotted] (0.,1.5)-- (10.,1.5);
\draw [dotted] (0.,1.)-- (10.,1.);
\draw [dotted] (0.,0.5)-- (10.,0.5);
\draw [line width=2.pt] (4.,8.)-- (0.,7.5);
\draw [line width=2.pt] (0.,7.5)-- (4.,7.);
\draw [line width=2.pt] (4.,7.)-- (0.,6.5);
\draw [line width=2.pt] (0.,6.5)-- (4.,6.);
\draw [line width=2.pt] (4.,6.)-- (0.,5.5);
\draw [line width=2.pt] (0.,5.5)-- (4.,5.);
\draw [line width=2.pt] (4.,5.)-- (0.,4.5);
\draw [line width=2.pt] (0.,4.5)-- (4.,4.);
\draw [line width=2.pt] (4.,4.)-- (0.,3.5);
\draw [line width=2.pt] (0.,3.5)-- (4.,3.);
\draw [line width=2.pt] (4.,3.)-- (0.,2.5);
\draw [line width=2.pt] (0.,2.5)-- (4.,2.);
\draw [line width=2.pt] (4.,2.)-- (0.,1.5);
\draw [line width=2.pt] (0.,1.5)-- (4.,1.);
\draw [line width=2.pt] (4.,1.)-- (0.,0.5);
\draw [line width=2.pt,dotted] (0.,0.5)-- (1.3981410439421973,0.2145171451162878);
\draw [line width=2.pt] (0.7889770769390853,0.3389007823364577)-- (0.,0.5);
\draw [line width=2.pt] (6.,7.)-- (10.,6.5);
\draw [line width=2.pt] (10.,6.5)-- (6.,6.);
\draw [line width=2.pt] (6.,7.)-- (10.,7.5);
\draw [line width=2.pt] (6.,6.)-- (10.,5.5);
\draw [line width=2.pt] (10.,5.5)-- (6.,5.);
\draw [line width=2.pt] (10.,8.)-- (10.,7.5);
\draw [line width=2.pt,dash pattern=on 1pt off 5pt on 9pt off 4pt] (6.,5.)-- (6.,4.5);
\draw [line width=2.pt,dash pattern=on 1pt off 5pt on 9pt off 4pt] (6.,4.5)-- (10.,4.);
\draw [line width=2.pt,dash pattern=on 1pt off 5pt on 9pt off 4pt] (10.,4.)-- (6.,3.5);
\draw [line width=2.pt,dash pattern=on 1pt off 5pt on 9pt off 4pt] (6.,3.5)-- (10.,3.);
\draw [line width=2.pt,dash pattern=on 1pt off 5pt on 9pt off 4pt] (10.,3.)-- (6.,2.5);
\draw [line width=2.pt,dash pattern=on 1pt off 5pt on 9pt off 4pt] (6.,2.5)-- (10.,2.);
\draw [line width=2.pt] (10.,2.)-- (10.,1.5);
\draw [line width=2.pt] (10.,1.5)-- (6.,1.);
\draw [line width=2.pt] (6.,1.)-- (10.,0.5);
\draw [line width=2.pt] (10.,0.5)-- (8.9709778184326,0.2982609957030703);
\draw [line width=2.pt,dotted] (8.9709778184326,0.2982609957030703)-- (8.420661086005177,0.20255373788960537);
\draw (1.74120541903861,0.013365674901894154) node[anchor=north west] {\emph{I}};
\draw (7.850446382873925,0.050845680814994486) node[anchor=north west] {\emph{K}};
\draw (-2.,8.)-- (-6.,8.);
\draw [dotted] (4.,7.5)-- (-6.,7.5);
\draw (-2.,8.)-- (-6.,8.);
\draw (-6.,8.)-- (-6.,0.);
\draw (-6.,0.)-- (-2.,0.);
\draw (-2.,0.)-- (-2.,8.);
\draw [line width=2.pt,dash pattern=on 1pt off 5pt on 9pt off 4pt] (-2.,7.)-- (-6.,6.5);
\draw [line width=2.pt,dash pattern=on 1pt off 5pt on 9pt off 4pt] (-6.,6.5)-- (-2.,6.);
\draw [line width=2.pt,dash pattern=on 1pt off 5pt on 9pt off 4pt] (-2.,7.)-- (-6.,7.5);
\draw [line width=2.pt,dash pattern=on 1pt off 5pt on 9pt off 4pt] (-2.,6.)-- (-6.,5.5);
\draw [line width=2.pt,dash pattern=on 1pt off 5pt on 9pt off 4pt] (-6.,5.5)-- (-2.,5.);
\draw [line width=2.pt,dash pattern=on 1pt off 5pt on 9pt off 4pt] (-6.,8.)-- (-6.,7.5);
\draw [line width=2.pt] (-2.,5.)-- (-2.,4.5);
\draw [line width=2.pt] (-2.,4.5)-- (-6.,4.);
\draw [line width=2.pt] (-6.,4.)-- (-2.,3.5);
\draw [line width=2.pt] (-2.,3.5)-- (-6.,3.);
\draw [line width=2.pt] (-6.,3.)-- (-2.,2.5);
\draw [line width=2.pt] (-2.,2.5)-- (-6.,2.);
\draw [line width=2.pt,dash pattern=on 1pt off 5pt on 9pt off 4pt] (-6.,2.)-- (-6.,1.5);
\draw [line width=2.pt,dash pattern=on 1pt off 5pt on 9pt off 4pt] (-6.,1.5)-- (-2.,1.);
\draw [line width=2.pt,dash pattern=on 1pt off 5pt on 9pt off 4pt] (-2.,1.)-- (-6.,0.5);
\draw [line width=2.pt,dash pattern=on 1pt off 5pt on 9pt off 4pt] (-6.,0.5)-- (-4.9709778184325994,0.29826099570306946);
\draw [line width=2.pt,dotted] (-4.9709778184325994,0.29826099570306946)-- (-4.420661086005177,0.2025537378896046);
\draw [dotted] (4.,7.)-- (-6.,7.);
\draw [dotted] (4.,6.5)-- (-6.,6.5);
\draw [dotted] (4.,5.5)-- (-6.,5.5);
\draw [dotted] (4.,5.)-- (-6.,5.);
\draw [dotted] (4.,4.)-- (-6.,4.);
\draw [dotted] (4.,4.5)-- (-6.,4.5);
\draw [dotted] (4.,3.5)-- (-6.,3.5);
\draw [dotted] (4.,3.)-- (-6.,3.);
\draw [dotted] (4.,2.5)-- (-6.,2.5);
\draw [dotted] (4.,2.)-- (-6.,2.);
\draw [dotted] (4.,1.)-- (-6.,1.);
\draw [dotted] (4.,1.5)-- (-6.,1.5);
\draw [dotted] (4.,0.5)-- (-6.,0.5);
\draw [dotted] (4.,6.)-- (-6.,6.);
\draw (-4.392330128665401,0) node[anchor=north west] {\emph{J}};
\draw (0.006210128263798415,0.05181361300366393) node[anchor=north west] {0};
\draw (3.6850916206612784,0.05181361300366393) node[anchor=north west] {1};
\draw (5.937963645149832,0.05181361300366393) node[anchor=north west] {2};
\draw (9.661167412776057,0.03325437141324524) node[anchor=north west] {3};
\draw (-2.3208988625864294,0.03325437141324524) node[anchor=north west] {-1};
\draw (-6.0884249054413985,0.05181361300366393) node[anchor=north west] {-2};
\draw (-6.4554580797551655,8.25084673847412) node[anchor=north west] {1};
\end{tikzpicture}

 \caption{Movement of points $x=[1,1]$, $y=[3,1]$, $z=[-2,1]$. Dashed parts of trajectories of $y$ and $z$ indicate when they are synchronic with $x$ and solid parts indicate when they are asynchronic. }
   \label{fig:2}
\end{figure}
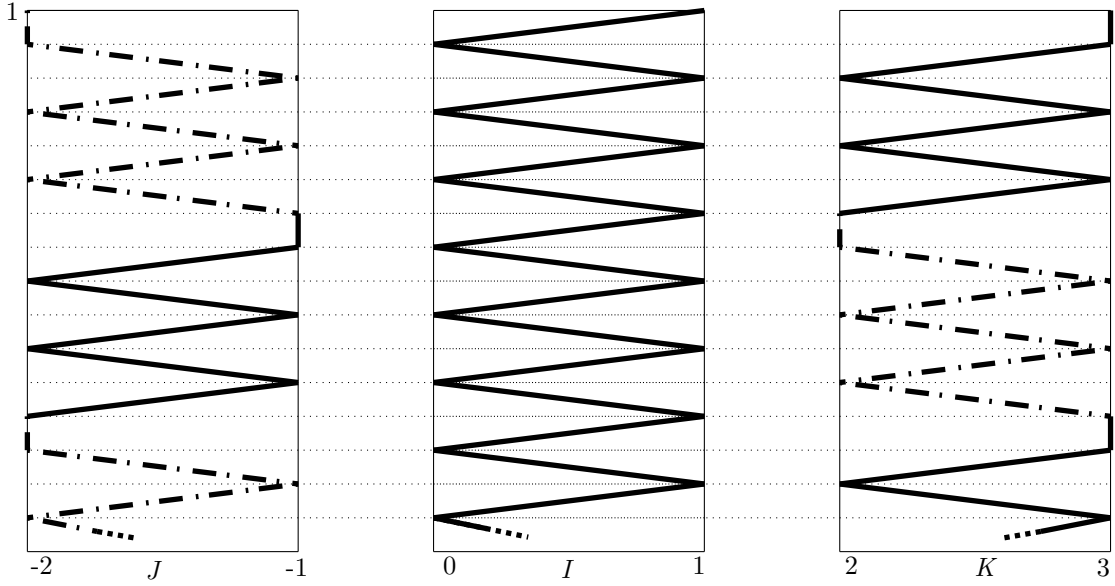

The idea of dynamical system $O$ is represented in Figure \ref{fig:2}. By (\ref{definition_H}), the oscillator above $K$ is mapped onto the oscillator above $J$ and vice versa. Moreover, these oscillators and their movement are reflection of each other with the line of symmetry $S=\frac{1}{2}$. It is easy to see that $H^{2}(x)=G^{2}(x)$, for $x\in K\times (\{\frac{1}{k}:k\in\mathbb{N}\}\cup\{0\})$, therefore existence of a DC3 pair for $G^2$ implies the same for $H^2$.\\
There are two types of points in $O$ - fixed in $I\cup K\cup J$ and oscillating. Fixed points can't be part of any DC3 pair by  arguments given in previous sections. Oscillating points are either $x=[1,1]$, $y=[3,1]$ and its mirror image $z=[-2,1]$, or points which are asymptotic or eventually equal to them. Therefore it is sufficient to investigate distribution functions among $x,y$ and $z$.\\
Denote the upper and lower distribution functions with respect to $H$ by $\Psi$ and $\Psi^*$. Pair $(y,z)$ is \emph{asynchronic} for the whole time - distance between $y$ and $z$ ranges from 3 to 5. By the similar argument as in~(\ref{odd}), 
 \begin{equation}\label{asyn} \Psi^*_{(y,z)}(\delta)=\Psi_{(y,z)}(\delta)=\lim_{m\to\infty}\frac{1}{s_m}\{ 0<i\leq s_m, d(H^i(x), H^i(y))<\delta\}=
 \left\{
  \begin{array}{l l}
    0&\quad \delta\leq1\\
   \frac{\delta-3}{2}& \quad 3<\delta\leq 5\\
   1&\quad \delta> 5.\\ 
  \end{array} \right.
  \end{equation} 
 We proceed with calculation of distribution function of $(x,y)$. In time interval $(s_m+2m, s_{m+1})$, where $m$ is even, the point $y$ is for half times above $K$ and $(x,y)$ is \emph{synchronic} - see Figure \ref{fig:2}. The other half times is $y$ above $J$, when $(x,y)$ is \emph{asynchronic}. 
Therefore we can use $\Phi^e_{(x,y)}$ (as a result of synchronic movement) and $\Phi^o_{(x,y)}$ (as a result of asynchronic movement) from (\ref{even}) and (\ref{odd}) to calculate the distribution function $\Psi^e_{(x,y)}$ as the arithmetic average 
\begin{equation} \Psi^e_{(x,y)}(\delta)=\lim_{\substack{m\to\infty \\ m \text{ is even}}}\frac{1}{s_m+2mn_m}\{ 0<i\leq s_m+2mn_m, d(H^i(x), H^i(y))<\delta\}=\frac{\Phi^e_{(x,y)}(\delta)+\Phi^o_{(x,y)}(\delta)}{2}.\end{equation} 
Similarly, in time interval $(s_m+2m, s_{m+1})$, where $m$ is odd, the point $y$ is for half times above $K$ and $(x,y)$ is \emph{asynchronic}. The other half times is $y$ above $J$, when $(x,y)$ is \emph{synchronic}. Hence
\begin{equation} \Psi^o_{(x,y)}(\delta)=\lim_{\substack{m\to\infty \\ m \text{ is odd}}}\frac{1}{s_m+2mn_m}\{ 0<i\leq s_m+2mn_m, d(H^i(x), H^i(y))<\delta\}=\frac{\Phi^e_{(x,y})(\delta)+\Phi^o_{(x,y)}(\delta)}{2},\end{equation} 
which shows $ \Psi^e_{(x,y)}= \Psi^o_{(x,y)}$.  We conclude that $$\Psi_{(x,y)}=\Psi^*_{(x,y)}=\frac{\Phi^e_{(x,y)}+\Phi^o_{(x,y)}}{2}.$$
Since $z$ is a mirror image of $y$, $(x,z)$ has the same distribution functions,
$$\Psi_{(x,z)}=\Psi^*_{(x,z)}=\frac{\Phi^e_{(x,y)}+\Phi^o_{(x,y)}}{2},$$  
therefore there are no DC3 pairs with respect to $H$. \qed 
\section{Iteration problem for $DC2\frac{1}{2}$}
\begin{theorem} For any integer $N>1$, function $f^N$ is distributionally chaotic of type $2\frac{1}{2}$ if and only if $f$ is too.
\end{theorem}
Theorem 2 follows by Lemmas 2 and 3 below. Lemma 1 is technical and follows by uniform continuity of $f$. Lemma 2 shows that if $(x,y)$ is DC$2\frac{1}{2}$ with respect to $f^n$ then $(x,y)$ is DC$2\frac{1}{2}$ with respect to $f$, while the opposite implication is proved by Lemma 3.\\
For a given function $f$, integer $N$ and two points $x,y$ in $X$, denote the distribution functions with respect to $f$ by $\Phi$
$$\Phi(\delta)=\liminf_{k\to\infty}\frac{1}{k}\#\{0\leq i<k; d(f^i(x),f^i(y))<\delta\},$$
$$\Phi^*(\delta)=\liminf_{k\to\infty}\frac{1}{k}\#\{0\leq i<k; d(f^i(x),f^i(y))<\delta\},$$
and with respect to $f^n$ by $\Psi$
$$\Psi(\delta)=\liminf_{k\to\infty}\frac{1}{k}\#\{0\leq i<k; d(f^{iN}(x),f^{iN}(y))<\delta\},$$
$$\Psi^*(\delta)=\liminf_{k\to\infty}\frac{1}{k}\#\{0\leq i<k; d(f^{iN}(x),f^{iN}(y))<\delta\}.$$
\begin{lemma}  For every $s>0$, there are numbers $t_1>0,t_2>0$ such that 
\begin{align*}
(i)&\ \Psi^*(t_1)\leq \Phi^*(s),&
(ii)&\Psi(t_1)\leq \Phi(s),&\\
(iii)&\Phi(t_2)\leq \Psi(s),&
(iv)&\Phi^*(t_2)\leq \Psi^*(s).&
\end{align*}

\end{lemma}
\begin{proof}
Denote $$\xi_n(f,s)=\#\{0\leq i<n: d(f^i(x),f^i(y))<s\}$$ and $$\delta_n(f,s)=\#\{0\leq i<n: d(f^i(x),f^i(y))\geq s\}.$$
(i) By uniform continuity of $f$, for any $s>0$ there is $t_1>0$ such that $d(x,y)<t_1$ implies $d(f^i(x),f^i(y))< s$, for $i=0,1,2,\ldots,N-1$. Then
$$N\cdot\xi_{n}(f^N,t_1)\leq \xi_{N\cdot n}(f,s),$$
or equivalently
$$\frac{1}{n}\xi_{n}(f^N,t_1)\leq \frac{1}{N\cdot n}\xi_{N\cdot n}(f,s).$$
After taking the limit superior of both sides, we get $\Psi^*(t_1)\leq \Phi^*(s)$, since limit superior of the right side is less or equal to $\Phi^*(s)$ by the definition of upper distribution function.\\
(ii) Since $d(x,y)<t_1$ implies $d(f^i(x),f^i(y))< s$, for $i=0,1,2,\ldots,N-1$,
$$N\cdot\xi_{\frac{n}{N}}(f^N,t_1)\leq\xi_n(f,s).$$
 After dividing by $n$ and taking the limit inferior of both sides,
$$\liminf_{n\to\infty}\frac{1}{\frac{n}{N}}\xi_{\frac{n}{N}}(f^N,t_1)\leq \Phi(s),$$
where the left side is greater or equal to $\Psi(t_1)$ by the definition of the lower distribution function, which finishes the proof of the second claim. \\
(iii) Since $f$ is uniformly continuous, there exists $t_2>0$ such that $d(f^N(x),f^N(y))\geq s$ implies $d(f^i(x),f^i(y))\geq t_2$, for $i=1,2,\ldots,N$. Therefore
$$N\cdot(\delta_{n}(f^N,s)-1)\leq \delta_{N\cdot n}(f,t_2),$$
or equivalently
$$1-\frac{1}{n}\delta_{n}(f^N,s)+\frac{1}{n}\geq 1-\frac{1}{N\cdot n}\delta_{N\cdot n}(f,t_2).$$
Since $\frac{1}{n}\xi_{n}(f^N,s)+\frac{1}{n}\delta_{n}(f^N,s)=1$, and similarly for $f$,
$$\frac{1}{n}\xi_{n}(f^N,s)+\frac{1}{n}\geq \frac{1}{N\cdot n}\xi_{N\cdot n}(f,t_2).$$
After taking the limit inferior of both sides, we get $\Psi(s)\geq \Phi(t_2)$, since limit inferior of the right side is greater or equal to $\Phi(t_2)$ by the definition of lower distribution function.\\
(iv) Since $d(f^N(x),f^N(y))\geq s$ implies $d(f^i(x),f^i(y))\geq t_2$, for $i=1,2,\ldots,N$,
$$N\cdot(\delta_{\frac{n}{N}}(f^N,s)-1)\leq\delta_n(f,t_2),$$ or equivalently
\begin{equation}\label{lemma}1-\frac{1}{\frac{n}{N}}\delta_{\frac{n}{N}}(f^N,s)+\frac{1}{\frac{n}{N}} \geq 1-\frac{1}{n}\delta_n(f,t_2).\end{equation}
Since $\frac{1}{n}\xi_{n}(f,t_2)+\frac{1}{n}\delta_{n}(f,t_2)=1$, and similarly for $f^N$,
 $$\frac{1}{\frac{n}{N}}\xi_{\frac{n}{N}}(f^N,s)+\frac{1}{\frac{n}{N}}\geq \frac{1}{n}\xi_n(f,t_2).$$
Taking the limit superior of both sides,
$$\limsup_{n\to\infty}\frac{1}{\frac{n}{N}}\xi_{\frac{n}{N}}(f^N,s)\geq \Phi^*(t_2),$$
where the left side is less or equal to $\Psi^*(s)$ by the definition of the upper distribution function.
\end{proof}
\begin{lemma} If there are $p>0$ and $c>0$ such that $\Psi^*(\delta)>c>\Psi(\delta)$, for any $0<\delta\leq p$, then there is $q>0$ such that
$\Phi^*(s)>c>\Phi(s)$, for any $0<s\leq q$.
\end{lemma}
\begin{proof}
By Lemma 1 (i), for any $s>0$ there is $t_1>0$ such that $\Phi^*(s)\geq \Psi^*(t_1)>c$ whence $\Phi^*(s)>c$, for any $s>0.$\\
By Lemma 1 (iii) there is a $q>0$ such that $\Phi(q)\leq \Psi (p)< c$. Since $\Phi$ is nondecreasing, we have $\Phi(s)\leq\Phi(q)$ whenever $0< s\leq q$.
\end{proof}
\begin{lemma} If there are $p>0$ and $c>0$ such that $\Phi^*(\delta)>c>\Phi(\delta)$, for any $0<\delta\leq p$, then there is $q>0$ such that
$\Psi^*(s)>c>\Psi(s)$, for any $0<s\leq q$.
\end{lemma}
\begin{proof}
By Lemma 1 (iv), for any $s>0$ there is $t_2>0$ such that $\Psi^*(s)\geq \Phi^*(t_2)>c$ whence $\Psi^*(s)>c$, for any $s>0.$\\
By Lemma 1 (ii) there is $q>0$ such that $\Psi(q)\leq \Phi (p)< c$. Since $\Psi$ is nondecreasing, $\Psi(s)\leq\Psi(q)$, for $0< \delta\leq q$.
\end{proof}
\section{Chaotic sets for iterated function}
Authors in \cite{Li} and \cite{W} understood the distributional chaos as the existence of a chaotic pair and we assumed the same in previous sections. Nowadays chaos is usually defined as the existence of an uncountable chaotic set. That arises a natural question - are results for the existence of uncountable chaotic sets for iterated function the same as for the existence of pairs? \\
We discovered that in case of DC1, DC2 and DC$2\frac{1}{2}$, the answer is easy - if $(x,y)$ is DC1 or DC2 with respect to $f$ then it is the same with respect to $f^n$. Therefore the existence of uncountable distributionally chaotic set with respect to $f$ ensures the same with respect to $f^n$. But for distributional chaos of type 3 is the situation more complicated - if $(x,y)$ is DC3 with respect to $f$ then there is $j\in\{0,1,\ldots, n-1\}$ such that $(f^j(x),f^j(y))$ is DC3 with respect to $f^n$. This $j$ can be different for different pairs in chaotic set, hence  the original chaotic set can split into chaotic pairs or into chaotic sets with smaller cardinality.\\
Let $S$ be a DC3 set with respect to $f$. We can generate an undirected graph $G$ in the following way - the set of vertices of $G$ is labeled by all points in $S$ and we add an edge between vertices $x$ and $y$ if $(x,y)$ is a chaotic pair. Then for a fixed $x\in S$ there is exactly one edge leading to every $y\in S\setminus\{x\}$ - hence $G$ is a complete graph. Next we assign a color $c_j$ to each number $j\in\{0,1,\ldots, n-1\}$. We color the graph $G$ with colors $c_0,\ldots,c_{n-1}$ in such a way that edge between $x$ and $y$ has color $c_j$ if $(f^j(x),f^j(y))$ is chaotic pair with respect to $f^n$. By our previous results about DC$2\frac{1}{2}$ pairs and by results about DC3 pairs in \cite{JD}, there is always at least one such $j$ (in case of multiple choices for $j$ we pick one randomly). Hence graph $G$ was colored by $n$ colors and we can use Ramsey theory to find a complete monochromatic subgraph which will represent a chaotic set with respect to $f^n$.\\
Let us recall a classic result from \cite{R}, reformulated for our purposes:
\begin{theorem} Let $G$ be a complete graph with infinite set of vertices and let each edge in this graph be colored by exactly one of colors $c_0,\ldots,c_{n-1}$. Then $G$ contains an infinite subgraph $H$ such that edges between every two distinct vertices in $H$ has the same color $c_i$, for some $i\in\{0,1,\ldots, n-1\}$.
\end{theorem}
An immediate consequence of this Infinite Ramsey Theorem is the following Corollary:
\begin{corollary} Let S be an infinite distributionally chaotic set of type 3 with respect to $f$. Then there is an infinite subset $R\subset S$ such that $f^j(R)$ is distributionally chaotic set of type 3 with respect to $f^n$, where $j\in\{0,1,\ldots, n-1\}$.
\end{corollary}
Unfortunately in case of infinite \emph{uncountable} graphs is the existence of an \emph{uncountable} monochromatic subgraph not ensured - Sierpinski coloring serves as an example in \cite{S}. Thus we pose an open question:\\

{\bf Question} \emph{Does the existence of an uncountable distributionally chaotic set of type 3 with respect to $f$ imply the same with respect to $f^n$?}

\paragraph{\bf Acknowledgment}
I sincerely thank my supervisor, Professor Jaroslav Sm\' ital, for valuable guidance. I~am grateful for his constant support and help.

\end{document}